\documentclass[12pt,leqno]{amsart}
\makeatletter
\AtBeginDocument{%
  \let\original@@tocwrite=\@tocwrite
  \newif\ifAHVflag
  \def\jlreq@uniqtoken{\jlreq@uniqtoken}
  \def\jlreq@endmark{\jlreq@endmark}
  \long\def\jlreq@getfirsttoken#1#{\jlreq@getfirsttoken@#1\bgroup\jlreq@endmark}
  \long\def\jlreq@getfirsttoken@#1#2\jlreq@endmark#3\jlreq@endmark{#1}
  \renewcommand{\@tocwrite}[2]{%
    \begingroup
      \AHVflagfalse 
      \@ifempty{#2}{}{%
        \expandafter\expandafter\expandafter\ifx\jlreq@getfirsttoken#2\jlreq@uniqtoken{}\jlreq@endmark\Sectionformat\expandafter\@firstoftwo\else\expandafter\@secondoftwo\fi
        {%
          \def\Sectionformat##1##2{\@ifempty{##1}{}{\AHVflagtrue}}%
          #2
        }{\AHVflagtrue}%
       }%
      \def\@tempa{}%
      \ifAHVflag\def\@tempa{\original@@tocwrite{#1}{#2}}\fi
    \expandafter\endgroup
    \@tempa
  }%
}
\makeatother

\usepackage{a4wide}
\usepackage{amssymb}
\usepackage{amsmath}
\usepackage{amsthm}
\usepackage[all]{xy}
\usepackage[usenames,dvipsnames]{xcolor}
\usepackage{lmodern}
\usepackage[T1]{fontenc}

\newcommand{\charf}{\operatorname{char}}

\newcommand{\Hom}{\operatorname{Hom}}

\newcommand{\Ind}{\operatorname{Ind}}

\newcommand{\ind}{\operatorname{ind}}

\newcommand{\End}{\operatorname{End}}

\DeclareMathOperator{\SL}{SL}

    \DeclareMathOperator{\Sym}{Sym}

   \DeclareMathOperator{\nrd}{nrd}

\theoremstyle{plain} 
\newtheorem{theorem}{Th\'eor\`eme}[section]
\newtheorem{corollary}[theorem]{Corollaire}
\newtheorem{lemma}[theorem]{Lemme}
\newtheorem{proposition}[theorem]{Proposition}

\newtheorem{definition-proposition}[theorem]{Definition-Proposition}
\newtheorem{variant}[theorem]{Variante}

\theoremstyle{definition}
\newtheorem{definition}[theorem]{D\'efinition}

\theoremstyle{remark}
\newtheorem{remark}[theorem]{Remarque}

\numberwithin{equation}{section}

\title{ Repr\'esentations  des quaternions de norme $1$  }
 
  \author{Guy Henniart and  Marie-France Vign\'eras  
}
 \date{\today}
\begin{document}

\begin{abstract} Let $p$ be a prime number, $F$ a local field with
finite residue field of characteristic $p$, $D$ the quaternion division algebra with centre $F$, and $R$   an 
algebraically  closed field of any characteristic $\charf_R$. We classify the smooth irreducible $R$-representations $\pi$
of the group $D^1$ of elements of $D^*$ with reduced norm $1$ to $F$. Such a $\pi $ occurs in the restriction of a smooth irreducible
$R$-representation $\Pi$ of $D^*$. When the dimension of $\Pi$ is $>1$,  following our previous work
in the case of $SL_2(F)$, we show that the restriction of $\Pi$ to $D^1$ is irreducible or
 the sum of two irreducible   representations.  When $\charf_R \neq p$, that restriction  is the sum of two irreducible  equivalent representations     if and only if  the representation   of $GL_2(F)$ corresponding to $\Pi$ via the 
Jacquet-Langlands correspondence restricts to $SL_2(F)$ as a sum of four inequivalent irreducible representations (this  is never the case  if $\charf_R \neq 2$).
  \end{abstract}

   \maketitle
 \def\contentsname{Plan}
\setcounter{tocdepth}{2}  
\tableofcontents
 \section{Introduction}
  Cette note est un compl\'ement \`a \cite{HV25}. Nous fixons un nombre premier $p$ et un corps  local non archim\'edien $F$ de caract\'eristique r\'esiduelle $p$. Nous fixons \'egalement un corps
  alg\'ebriquement clos $R$. Lorsque la caract\'eristique $\charf_R$ de $R$ n'est pas $p$, dans loc.cit., nous avons classifi\'e les $R$-repr\'esentations  lisses irr\'eductibles de $SL_2(F)$, en utilisant qu'elles apparaissent dans la restriction de $R$-repr\'esentations lisses irr\'eductibles de $GL_2(F)$. Dans le pr\'esent article, nous consid\'erons un corps de quaternions $D$ de centre $F$, $\nrd:D^*\to F^*$ la norme r\'eduite et le groupe $D^1$ form\'e des \'el\'ements  dont la norme r\'eduite   vaut $1$. C'est une forme int\'erieure non d\'eploy\'ee de $SL_2(F)$, unique \`a isomorphisme pr\`es. Nous classifions les $R$-repr\'esentations lisses
 irr\'eductibles de $D^1$. Elles apparaissent dans la restriction des $R$-repr\'esentations lisses irr\'eductibles de $D^*$, sans hypoth\`ese sur $\charf_R$.
 Notre r\'esultat principal est le suivant.
 
 \begin{theorem} 1)  La restriction   \`a $D^1$ d'une $R$-repr\'esentation lisse irr\'eductible $\Pi$ de $D^*$ de dimension $>1$ est irr\'eductible ou somme de deux repr\'esentations irr\'eductibles.
 
 2) Si  $\charf_R=p$,    les $R$-repr\'esentations lisses irr\'eductibles de  $D^1$ sont les $R$-caract\`eres de $D^1$.  Ces caract\`eres  forment un groupe cyclique d'ordre $q+1$ o\`u $q$ est le cardinal du corps r\'esiduel de $F$.
  
3) Si  $\charf_R\neq p$, nous notons $JL(\Pi)$  la $R$-repr\'esentation irr\'eductible de $GL_2(F)$ correspondant \`a $\Pi $ par la correspondance de Jacquet-Langlands, avec $R$ comme corps de coefficients. Alors   la restriction de $JL(\Pi)$ \`a $SL_2(F)$ est somme de $d_\Pi$  repr\'esentations irr\'eductibles non \'equivalentes, o\`u  $d_\Pi$  est la dimension de  l'alg\`ebre  des endomorphismes de $\Pi|_{D^1}$
  \end{theorem}
 
   L'alg\`ebre $\End_{RD^1}\Pi$ des endomorphismes de $\Pi|_{D^1}$ est de dimension $d_\Pi=1$ si $\Pi|_{D^1}$ est irr\'eductible,  $d_\Pi=2$   si $\Pi|_{D^1}$ a deux  composants irr\'eductibles  non \'equivalents,  $d_\Pi=4$    si $\Pi|_{D^1}$ a deux composants irr\'eductibles   \'equivalents.

Lorsque  $\charf_R\neq p$,  la diff\'erence avec le cas de $SL_2(F)$ r\'eside dans le cas o\`u les deux repr\'esentations irr\'eductibles de $D^1$ dans $\Pi|_{D^1}$ sont \'equivalentes. 
Ci-apr\`es, nous fixons les notations et \'enon\c{c}ons quelques rappels dont nous aurons besoin, puis passons directement \`a la d\'emonstration du Th\'eor\`eme. Elle se fait en deux temps, au \S \ref{pre}.

Si  $\charf_R\neq 2, p$,  $d_\Pi $   est le nombre des $R$-caract\`eres lisses $\chi$ de  $F^*$ tels que la torsion par $\chi \circ  \nrd$  stabilise $\Pi$. On utilisera alors que la correspondance de Langlands qui associe
\`a $\Pi$ une $R$-repr\'esentation irr\'eductible lisse  de dimension $2$ du groupe de Weil absolu $W_F$ de $F$ est compatible \`a la torsion par les caract\`eres, ce qui permet d'utiliser les r\'esultats de loc.cit..
 
 Si  $\charf_R=2$ ou $p$,  nous utiliserons  la construction explicite de $\Pi$ (voir \cite{V87}).   Dans ce cas, nous verrons   en \S \ref{ss:31}, 
que $d_\Pi=2$ sauf si 
 $\charf_R=p$, $p$ est impair et $\Pi$ est induite d'un $R$-caract\`ere $\lambda$ de $F^*U_D $ de restriction \`a $D^1$ 
 l'unique $R$-caract\`ere d'ordre $2$,  o\`u $U_D$ est le groupe des unit\'es de l'anneau des entiers de $D$.

\bigskip Soit $\ell$ un nombre premier. Toute $\mathbb Q_\ell^{ac}$-repr\'esentation\footnote{On note $X^{ac}$  une cl\^oture alg\'ebrique d'un corps commutatif $X$} lisse irr\'eductible de $D^1$ est enti\`ere car $D^1$ est compact. La r\'eduction  modulo $\ell$ est aussi simple que possible pour les repr\'esentations de   $D^1$. Nous prouverons au \S \ref{s:red}:

\begin{theorem} 1) La r\'eduction modulo $\ell$ d'une $\mathbb Q_\ell ^{ac}$-repr\'esentation  irr\'eductible  lisse  de $D^1$ est toujours irr\'eductible (ce n'est pas vrai pour $D^*$ et si $\ell \neq p$ pour $GL_2(F), SL_2(F)$).
 
2) Toute $\mathbb F_\ell ^{ac}$-repr\'esentation  irr\'eductible   lisse  de $D^1$  est la r\'eduction modulo $\ell$ d'une $\mathbb Q_\ell ^{ac}$-repr\'esentation  irr\'eductible  lisse  de $D^1$ (c'est  vrai pour $D^*$ et si $\ell \neq p$ pour $GL_2(F), SL_2(F)$).
 \end{theorem}

 Enfin au \S \ref{CLE},   nous examinons la possibilit\'e d'une correspondance de Langlands \'etendue qui \`a $\Pi$ associe un
 morphisme $\phi$ de $W_F$ dans $PGL_2(R)$ et un caract\`ere du groupe des composantes connexes du centralisateur de $\phi$.
 
 \bigskip Dans l'appendice \S\ref{irr}, nous donnons des crit\`eres d'irr\'eductibilit\'e, utilis\'es au \S \ref{pre}, mais d'un int\'er\^et plus g\'en\'eral.
Dans l'autre appendice \S\ref{com}, nous consid\'erons une $F$-alg\`ebre \`a division centrale $D$, de degr\'e r\'eduit $d>1$, et le noyau $D^1$
de la norme r\'eduite de $D^* $ \`a $F^*$. Nous prouvons que tout \'el\'ement de $D^1  \cap  (1+P_D)$ (o\`u $P_D$ est l'id\'eal maximal 
de l'anneau des entiers de $D$) est produit de deux commutateurs de $D^1$. En particulier $ D^1\cap (1+P_D)$ est le groupe
des commutateurs de $D^1$, un r\'esultat d\^u \`a C. Riehm.  Enfin dans l'appendice \S \ref{modp} dans le cas $F=\mathbb Q_p$ nous comparons nos r\'esultats avec ceux pour $SL_2(\mathbb Q_p)$ de \cite{Abde14} lorsque   $ R=\mathbb F_p^{ac}$, et de \cite{BS25} pour des repr\'esentations de Banach.
 
\section{Notations et rappels}
\subsection{}\label{not} $F$ est un corps local non-archim\'edien de corps r\'esiduel $k_F$ de caract\'eristique $p $ et de cardinal  $q$, $R$ est un corps alg\'ebriquement clos sauf au \S \ref{irr}, $D$ est un corps de quaternions de centre $F$, sauf au \S \ref{com}. On notera  $\mu_F$ le groupe cyclique des racines de l'unit\'e dans $F^*$ d'ordre divisant $q-1$,   $O_D$ l'anneau d'entiers de $D$, $P_D$  l'id\'eal maximal    de $O_D$, $\omega\in D$ une racine  d'ordre $q^2-1$ de d'unit\'e, $p_D$   un g\'en\'erateur de $P_D$ tel que $p_D^2=p_F, \ p_D \omega p_D^{-1}=\omega^q$, $U_D$ le groupe des unit\'es de $O_D$,  $k_D$ le corps r\'esiduel de $O_D$  (une extension quadratique de $k_F$), et $\varrho_D:O_D\to k_D$ la surjection canonique.
La conjugaison par $p_D$ induit l'\'el\'evation \`a la puissance $q$ sur $k_D$.

La norme r\'eduite $\nrd:D^* \to  F^*$ de noyau $D^1$  induit un  hom\'eomorphisme  $  D^*/F^*D^1  \to  F^*/(F^*)^2$. Le groupe
$F^*D^1$ est ouvert et cofini dans $D^* $ sauf si  $\charf_F=2$ o\`u   il est seulement ferm\'e cocompact. 

\subsection{}  La  correspondance locale de Langlands pour $D^*$  avec $R$ comme corps de coefficients,   donne une bijection 
  $\Pi \leftrightarrow \sigma(\Pi)$
   entre les classes d'isomorphisme des  $R$-repr\'esentations irr\'eductibles lisses 
  de dimension $>1$  de $D^*$ et celles  de dimension $2$ du groupe de Weil $W_F$ de $F$ \cite{V87}. Si $\charf_R\neq p$ la correspondance locale de Langlands pour $GL_2(F)$ donne une bijection  $\sigma(\Pi) \leftrightarrow JL(\Pi) $ entre les classes d'isomorphisme des  $R$-repr\'esentations de dimension $2$ de $W_F$  et celles  des $R$-repr\'esentations supercuspidales de $GL_2(F)$ \cite{V01}. Ces bijections sont compatibles \`a la torsion par les caract\`eres au sens suivant. Soit $\alpha:W_F\to F^*$ l'application de r\'eciprocit\'e
du corps de classes et $\det:GL_2(F)\to F^*$ le d\'eterminant. Pour tout $R$-caract\`ere lisse $\chi$ de $F^*$, 
$$\Pi \otimes \chi \circ \nrd \leftrightarrow \sigma(\Pi) \otimes \chi \circ \alpha  , \ \ \ \sigma(\Pi) \otimes \chi \circ \alpha\leftrightarrow  JL(\Pi) \otimes \chi \circ \det.$$
Ces bijections  d\'ependent du choix d'une racine carr\'ee de $q$ dans $R^*$, mais si   $\charf_R\neq p$,  
 la correspondance locale de Jacquet-Langlands $\Pi \leftrightarrow JL(\Pi)$
entre les classes d'isomorphisme des  $R$-repr\'esentations irr\'eductibles lisses 
   de dimension $>1$ de $D^*$ et  supercuspidales de $GL_2(F)$ n'en d\'epend pas.

\subsection{}
Soit $\ell$ un nombre premier. 
Une  $\mathbb Q_\ell^{ac}$-repr\'esentation  lisse irr\'eductible $\Pi$ de $D^*$  est enti\`ere si et seulement si son caract\`ere central $\omega_\Pi$ est entier.  Si elle est enti\`ere, la longueur de sa r\'eduction modulo $\ell$ est $\leq 2$. Quand elle est \'egale \`a $2$,  la dimension de $\Pi$  est $2$. C'est le cas si $\ell=p$. Toute $\mathbb F_\ell^{ac}$-repr\'esentation lisse irr\'eductible  de $D^*$  est la   r\'eduction modulo $\ell$ d'une $\mathbb Q_\ell^{ac}$-repr\'esentation lisse irr\'eductible de $D^*$   \cite{V87}.

\section{Preuve du th\'eor\`eme principal}\label{pre}
\subsection{}\label{ss:30}

Toute $R$-repr\'esentation lisse irr\'eductible de $ D^*$ ou de $D^1$ est de dimension finie  et a  un caract\`ere central \cite[\S2]{H09}. Une $R$-repr\'esentation lisse irr\'eductible $\pi$ de $ D^1$ apparait dans la restriction d'une $R$-repr\'esentation lisse irr\'eductible $\Pi$ de $ D^*$. En effet, $\Pi$ est triviale sur $1+P_D^i$ pour $i>>0$, donc $\pi$ s'\'etend \`a $F^*D^1(1+P_D^i)$ qui est d'indice fini dans $D^*$.
 La restriction de $\Pi$ \`a $D^1$ est semi-simple de longueur finie, de composants irr\'eductibles ayant la m\^eme multiplicit\'e \cite{HV25}. 

 \begin{definition} Une $R$-repr\'esentation irr\'eductible lisse $\Pi$ de $D^*$ d\'efinit un $L$-paquet $L(\Pi)$ de $D^1$,  form\'e de l'ensemble des classes d'isomorphismes des $R$-repr\'esentations irr\'eductibles contenues dans  sa restriction \`a $D^1$.  \end{definition}
 
 Deux $L$-paquets sont disjoints ou confondus, et l'union des $L$-paquets de  $D^1$ est l'ensemble des  classes d'isomorphisme des $R$-repr\'esentations irr\'eductibles de $D^1$.

 Comme  $\Pi$ a un caract\`ere central, on a $ \End_{RD^1} \Pi=  \End_{RF^*D^1} \Pi$. Comme dans  la preuve de   \cite[lemme 2.3]{HV25}
$$ \End_{RF^*D^1} \Pi\simeq   \Hom_{RD^*} (\Pi, \ind_{F^*D^1}^{D^*} \Pi|_{F^*D^1})\simeq \Hom_{RD^*} (\Pi,\Pi \otimes \ind_{F^*D^1}^{D^*} 1).$$
Notons $d_\Pi$  la dimension  de $ \End_{RD^1} \Pi$. 

 Si $d_\Pi=1$ alors  $\Pi|_{D^1}$ est irr\'eductible,  si $d_\Pi=2$ alors  $\Pi|_{D^1}$ est somme de deux repr\'esentations irr\'eductibles  non \'equivalentes,  si $d_\Pi=4$   alors $\Pi|_{D^1}$ est somme de deux repr\'esentations irr\'eductibles   \'equivalentes ou de quatre repr\'esentations irr\'eductibles  non \'equivalentes. 
 
 Nous allons montrer que $d_\Pi=1,2$ ou $4$ et que si $d_\Pi=4$ le dernier cas ne se produit pas.

\subsection{}  Le groupe des commutateurs de $D^*$ est $(D^*,D^*)=D^1$ \cite{NM43} donc les $R$-caract\`eres lisses de $D^*$ sont $\chi \circ \nrd$ pour les $R$-caract\`eres lisses $\chi$  de $F^*$. Leur restriction \`a $D^1$ est le caract\`ere trivial.
Mais $D^1$ poss\`ede des caract\`eres non triviaux car le commutateur de $D^1$ est $D^1 \cap (1+P_D)$ \cite{Rie70} (voir l'appendice \S \ref{com}). L'application $\rho_D$ donne un isomorphisme de $D^1/(D^1 \cap (1+P_D))$ sur le noyau $k_D^1$ de la norme de $k_D/k_F$. Ce noyau est cyclique d'ordre $q+1$. Donc  les $R$-caract\`eres lisses de $D^1$ sont $\chi \circ \rho_D$ pour les $R$-caract\`eres lisses  $\chi$ de $k_D^1$. 

\bigskip  Si $\charf_R=p$, toute $R$-repr\'esentation irr\'eductible lisse $\pi$ de $D^1$ est triviale sur le pro-$p$ sous groupe $D^1 \cap (1+P_D)$ donc est un caract\`ere. Ceci d\'emontre la partie 2) du th\'eor\`eme principal.

\bigskip Sans hypoth\`ese sur $\charf_R$, les caract\`eres  de $D^1$ apparaissent dans les restrictions \`a $D^1$ des $R$-repr\'esentations irr\'eductibles lisses $\Pi$ de $D^*$ triviales sur $1+P_D$. Une telle repr\'esentation de dimension $>1$ est de la forme \cite{V87}, \cite[\S3]{H09}
    $$\Pi \simeq \ind_{F^*U_D}^{D^*} \lambda$$
est induite d'un $R$-caract\`ere  $\lambda $ de $F^*U_D$ trivial sur $1+P_D$ tel que   $ \lambda |_{U_D}$ est r\'egulier, i.e. $ \lambda |_{U_D} =\nu \circ \rho_D|_{U_D}$
pour un caract\`ere $\nu$ de $k_D^*$  r\'egulier sur $k_F$, i.e. $\nu^q\neq \nu$.  

Comme $D^* = F^*U_D\sqcup p_D F^*U_D$,   la dimension de $\Pi$ est $2$ et la restriction de $\Pi$  \`a $D^1$ est la somme  de la restriction \`a $D^1$ de   $\lambda$ et de  son conjugu\'e   par $p_D$. 
 On a ${}^{p_D} \chi =\chi^q$ et  $\rho_D(D^1)$ est  le noyau $k_D^1$ de la norme de $k_D^*$ \`a $k_F^*$. La r\'egularit\'e  $\nu ^q \neq \nu$  est \'equivalente \`a la non-trivialit\'e de $\nu$ sur $k_D^1$, car  l'image de l'homomorphisme $x\to x^{ q-1}$ de noyau $k_F^*$ est  d'ordre $|k_D^*| /  |k_F^*|= q+1$, c'est donc $k_D^1$.  On note $\pi(\nu)=\lambda|_{D^1}= \nu|_{k_D^1} \circ \rho|_{D^1}$. Conjuguant par $p_D$, on a $\pi(\nu^q)= \lambda^{p_D}|_{D^1}= \nu^q|_{k_D^1} \circ \rho|_{D^1}$. On  obtient: 

\begin{lemma}\label{le:mr}  $\Pi|_{D^1} \simeq  \pi(\nu) \oplus \pi(\nu^q)$ est somme de deux caract\`eres.
\end{lemma}  
On a  donc  $d_\Pi=4$ si ces caract\`eres sont \'egaux et  $d_\Pi=2$ sinon.

\medskip  Quand a-t-on  l'\'egalit\'e $ \pi(\nu) =  \pi(\nu^q)$ ?   
L'application $\nu|_{k_D^1} \to \pi(\nu)$ est injective, donc  l'\'egalit\'e  est \'equivalent \`a $\nu^q=\nu$ sur $k_D^1$.  Le  $R$-caract\`ere $\nu|_{k_D^1}$ n'est pas trivial et $(\nu|_{k_D^1})^{q+1}=1$, aussi  $\nu|_{k_D^1}^{q-1}=1$ si et seulement si l'ordre de $\nu|_{k_D^1}$ est positif et divise  $(q-1,q+1)$ si et seulement si  $\nu|_{k_D^1}$ est d'ordre $2$.  Le groupe $k_D^1$ est cyclique d'ordre $q+1$. Il n'admet  pas de $R$-caract\`ere d'ordre $2$ si $p=2$ ou si $\charf_R=2$. Sinon, il admet un unique $R$-caract\`ere d'ordre $2$. On a donc:

 \begin{lemma}\label{le:2}  1) On a $\nu=\nu^q$ sur $k_D^1$ si et seulement si $\nu|_{k_D^1}$ est d'ordre $2$. 
 
2) On a  $d_\Pi=2$ sauf si  $p$ est impair, la caract\'eristique de $R$ n'est pas $2$, et $\nu|_{k_D^1}$ est l'unique  $R$-caract\`ere d'ordre $2$  de $k_D^1$, o\`u $d_\Pi=4$.

 \end{lemma}
  
\begin{remark} Si $\charf_R \neq p$ et $p$   impair, il existe un unique $L$-paquet de $D^1$ avec $d_\Pi=4$ par la  correspondance  de Jacquet-Langlands et \cite[Proposition 4.22]{HV25}. Le lemme \ref{le:2} le d\'ecrit explicitement. 
\end{remark} 

Si $\charf_R=p$, le th\'eor\`eme principal est d\'emontr\'e.
\subsection{}\label{ss:31}On  suppose que $\Pi$ est une $R$-repr\'esentation lisse irr\'eductible de $ D^*$ de dimension $>1$. 
 
On suppose aussi que $\charf_R \neq 2,  p$.  Alors la  $R$-repr\'esentation $\ind_{ZD^1}^{D^*} 1$ de $D^*$ est la somme des  $\chi \circ \nrd$ o\`u $\chi $ parcourt les $R$-caract\`eres lisses de $F^*$ de carr\'e trivial. 
 Le nombre de $\chi$ tels que 
$\Pi \simeq \Pi \otimes \chi \circ \nrd $ est $d_\Pi$.
Il n'y a pas de diff\'erence avec le cas de $SL_2(F)$. Le nombre de $\chi$ tels que 
$JL(\Pi )\simeq JL(\Pi )\otimes \chi \circ \det$  est \'egal \`a la  dimension de  $ \End_{RSL_2(F)} JL(\Pi)$ \cite{HV25}. 
La correspondance locale de  Langlands implique:
\begin{equation}\label{eq:QGG}\Pi \simeq \Pi \otimes \chi \circ \nrd \Leftrightarrow \sigma(\Pi) \simeq \sigma(\Pi)   \otimes \chi \circ\alpha   \Leftrightarrow  JL(\Pi)  \simeq  JL(\Pi)  \otimes \chi \circ\det.
\end{equation}
Donc, 
\begin{lemma} Les  alg\`ebres d'endomorphismes  de  $\Pi|_{D^1}$ et de $JL(\Pi)|_{SL_2(F)}$ ont la m\^eme dimension.
\end{lemma}
  $JL(\Pi)|_{SL_2(F)}$ est somme de $d_\Pi$  repr\'esentations irr\'eductibles non \'equivalentes, car la  restriction de $JL(\Pi)$ \`a $SL_2(F)$ est  sans multiplicit\'e   (loc. cit.). 
  
 \medskip  On  d\'eduit du lemme et de \cite[Theorem1.1]{HV25} que $d_\Pi=1,2$ ou $4$.

 \medskip Si $R=\mathbb C$, le dernier cas o\`u $d_\Pi=4$ et $\Pi|_{D^1}$  sans multiplicit\'e, ne se produit pas
 \cite{L71},  \cite{LL79}, \cite{L24}. Cette propri\'et\'e reste vraie  dans le cas $\ell$-adique $R=\mathbb Q_\ell^{ac}$ car $ \mathbb Q_\ell^{ac}\simeq \mathbb C$, et dans le cas $\ell$-modulaire $R=\mathbb F_\ell^{ac}$  car $\Pi$ se rel\`eve a $\mathbb Q_\ell^{ac} $ en une repr\'esentation $\tilde \Pi$ telle que $d_{\tilde \Pi}= d_\Pi$ (ceci a \'et\'e d\'emontr\'e pour $\sigma_\Pi$ \cite[Theorem 4.23 2)]{HV25}). 
 Cela reste vrai pour tout $R$  (suppos\'e toujours de caract\'eristique non $2$), car  $\Pi $ tordue par un caract\`ere lisse est d\'efinie sur  la cl\^oture alg\'ebrique $R_c$  de son corps premier, les $R$-caract\`eres lisses  de $F^*$ de carr\'e trivial sont des $R_c$-caract\`eres,  et l'extension des scalaires  de $R_c$ \`a $R$ donne une  injection des classes d'isomorphisme des repr\'esentations lisses irr\'eductibles de $D^*$ sur $R_c$ vers celles sur $R$.

\medskip    Ceci ach\`eve la preuve  du th\'eor\`eme principal si $\charf_R=p$  ou si $\charf_R \neq 2 $. Il reste le cas   $\charf_R=2$  et $p$ impair. On utilisera la construction explicite de $\Pi$ pour montrer le th\'eor\`eme (voir la fin de \S \ref{ss:direct}).

\subsection{}\label{ss:direct} Une $R$-repr\'esentation lisse irr\'eductible $\Pi$ de $ D^*$ 
est tordue  par un caract\`ere  d'une repr\'esentation  mod\'er\'ee ou  minimale sauvagement ramifi\'ee.  Supposons   $\charf_R \neq p$ et $\Pi$ minimale sauvagement ramifi\'ee. On a une description explicite de $\Pi$ \cite{V87}.

Soit $f>1$ l'entier tel que 
 $\Pi$ est triviale sur $1+P_D^f$ mais  non sur $1+P_D^{f-1}$. Le groupe $F^*(1+P_D^{[(f+1)/2]})$  est commutatif modulo $1+P_D^f$. On choisit un $R$-caract\`ere $\chi $ de $F^*(1+P_D^{[(f+1)/2]}) $ contenu dans  $\Pi|_{F^*(1+P_D^{[(f+1)/2]}) }$ et l'on note $J$ son centralisateur dans $D^* $.
Alors                               
$$ \Pi\simeq \ind_J^{D^*}\lambda $$
 est induite d'une $R$-repr\'esentation  $\lambda $ de $J$ de restriction $\chi$-isotypique \`a $ F^*(1+P_D^{[(f+1)/2]})$.
 Lorsque $f$ est pair,  
         $$  J=E^* (1+P_D^{f/2}),  \ E/F  \ \text{quadratique,   s\'eparable, ramifi\'ee}$$
        (a priori la construction ne donne pas $E/F$ s\'eparable, mais on peut choisir $E/F$ s\'eparable), 
et $J/Ker(\chi)$ est ab\'elien, $\lambda $ est un caract\`ere,  la dimension de $\Pi $ est 
 $$[D^*:J]=[U_D: O_E^* (1+P_D^{f/2})]= (q+1) \, [1+P_D: (1+P_E) (1+P_D^{f/2})] >2.$$ 
 Lorsque $f $ est impair,  
$$J=E^* (1+P_D^{(f-1)/2}) ,  \ E/F  \ \text{quadratique,  s\'eparable, non ramifi\'ee}, $$
et $J/Ker(\chi)$ est un groupe d'Heisenberg.  Le groupe $J=F^*O_E^*(1+P_D^{[f/2]})$ contient les  sous-groupes distingu\'es
$$J''=F^*(1+P_E) (1+P_D^{(f+1)/2}) \subset J'=F^*(1+P_E) (1+P_D^{(f-1)/2}),$$    $J/Ker(\chi)$ est  une extension centrale par le groupe commutatif $J'' /Ker (\chi)$ du  groupe fini $J'/J''$ d'ordre $q^2$. 
La restriction de $\lambda $ \`a $J'$ est irr\'eductible, la dimension de $\lambda$ est $q$, 
 celle de $\Pi$ est 
$$q[D^*:J]=q [D^*:E^* (1+P_D^{(f-1)/2})]=2q \, [1+P_D: (1+P_E)(1+P_D^{(f-1)/2}]> 2.$$
L'indice de $JD^1$ dans $D^*$  \'egal \`a $[F^*: \nrd(J)]$ divise $2$
 car $ \nrd(J)$ contient   $N_{E/F}(E^*)$. On note $$\pi(J,\lambda)=\ind_{J\cap D^1}^{D^1}(\lambda|_{J\cap D^1}),$$ 
 Si $D^*=JD^1$ alors $\Pi|_{D^1}= \pi(J,\lambda)$. Si $p=2$ on peut avoir $D^*=JD^1$  et  $\pi(J,\lambda)$ irr\'eductible.  
 \begin{lemma} Si  $p=2$ et $\Pi$ triviale sur $1+P_D^2$ mais non sur $1+P_D$, alors $\Pi|_{D^1}$ est irr\'eductible. 
 \end{lemma}
 \begin{proof}[Preuve]
Comme $f=2$ est pair,   la dimension de  $\Pi$   est  $[D^*:J]=q+1$ et $E/F$ est ramifi\'e,  Donc  la norme  r\'eduite de $E^*$ contient une uniformisante de $F$, et la norme  r\'eduite  de $\mu_F$ est $\mu_F^2$. On a   
 $ \mu_F^2=\mu_F$  et $\mu_F \cap D^1$ est trivial  car $p=2$.  On a  $1+P_F=\nrd(1+P_D)=\nrd(1+P_D^2)$ (voir la formule \eqref{eq:iso0}). Donc $F^*=\nrd (J)$, $D^*=JD^1$, et   $J\cap D^1= (1+P_D) \cap D^1$. Donc (voir la formule \eqref{eq:iso}),
 $$\frac{J\cap D^1}{(1+P_D^2)\cap D^1}=   \frac{(1+P_D)\cap D^1}{(1+P_D^2)\cap D^1} \simeq \frac{1+P_D}{1+P_D^2}.$$ 
  Pour $x\in O_D$ tel que $1+p_D x \in J\cap D^1$ modulo $1+P_D^2$ on  a $\lambda (1+p_D x)=\psi( \varrho(x))$ pour un $R$-caract\`ere non-trivial $\psi$ de $k_D$.
La repr\'esentation $\Pi|_{D^1}=\ind_{J\cap D^1}^{D^1}\lambda|_{J\cap D^1}$ est irr\'eductible 
    \end{proof}
 Supposons maintenant $[D^*:JD^1]=2$. C'est le cas,  si
  l'extension $E/F$ est non ramifi\'ee,  car   $JD^1$  ne contient pas $p_D$, ou si l'extension $E/F$ est  ramifi\'ee et $p$ est impair,   car $ \nrd(J)$ contient pas de  racine de l'unit\'e d'ordre $q-1$. 
 Pour $d$ dans $D^*$ non dans $JD^1$,   $$\Pi|_{D^1}= \pi(J,\lambda) \oplus  {}^d \pi(J,\lambda)$$ est somme de deux repr\'esentations de dimension  $(\dim_R \Pi)/2  >1$. 
  
  La restriction de $\lambda$  au pro-p radical $J_p= (1+P_E)(1+P_D^{[f/2]}) $  de $J=E^*(1+P_D^{[f/2]})$ est irr\'eductible.  L'entrelacement dans $D^*$ de $\chi$ est $J$,  donc celui de $\lambda$  et de $ \lambda|_{J_p}$  dans $D^*$ est aussi $J$.

 Si $p$ est impair, tout \'el\'ement de $1+P_F$  est un carr\'e, et $J_p=(Z\cap J_p)(D^1\cap J_p)$ car $\nrd(J_p)\subset 1+P_F=Z\cap J_p= (1+P_F)^2 = \nrd (Z\cap J_p) \subset \nrd(J_p)$. 
 Comme $Z\cap J_p$ agit dans $\lambda$ par un caract\`ere, 
 $\lambda|_{D^1\cap J_p}$ est  irr\'eductible, 
d'entrelacement   dans $D^*$  \'egal \`a $J$, donc  l'entrelacement dans  $D^1$ est  $J \cap D^1$. Donc $\pi(J,\lambda)$ est irr\'eductible (Remarque \ref{rem:irr}).
Il  existe $d$ dans $D^*$ non dans $JD^1$
car $p$ est impair, donc   il n'existe pas 
d'entrelacement de $\lambda|_{J \cap D^1}$ avec son conjugu\'e par $d$, donc  les deux composants $\pi(J,\lambda)$  et $ {}^d \pi(J,\lambda)$
 de $\Pi |_{ D^1}$ sont   irr\'eductibles et non \'equivalents. Il en d\'ecoule:
 
\begin{lemma}\label{le:srm}  Si $p$ est impair, alors  $\Pi |_{ D^1}$ est somme de deux $R$-repr\'esentations irr\'eductibles non \'equivalentes  de dimension $>1$.
\end{lemma}

Si  $\charf_R=2 $ et  $p$  impair, le th\'eor\`eme principal est aussi d\'emontr\'e dans ce cas, car 
   la  $R$-repr\'esentation $JL(\Pi)|_{SL_2(F)}$  est aussi  somme de deux $R$-repr\'esentations irr\'eductibles non isomorphes  (voir \cite[Theorem 1.5]{HV25}).

\section{R\'eduction modulo $\ell$}\label{s:red}  
Toute $\mathbb Q_\ell^{ac}$-repr\'esentation lisse irr\'eductible de $D^1$ est enti\`ere car $D^1$ est compact. Sa  r\'eduction  modulo $\ell$  est aussi simple que possible.

\begin{theorem} 1) La r\'eduction modulo $\ell$ d'une $\mathbb Q_\ell ^{ac}$-repr\'esentation  irr\'eductible  lisse  de $D^1$ est toujours irr\'eductible (ce n'est pas vrai pour $D^*, GL_2(F), SL_2(F)$).
 
2) Toute $\mathbb F_\ell ^{ac}$-repr\'esentation  irr\'eductible   lisse  de $D^1$  est la r\'eduction modulo $\ell$ d'une $\mathbb Q_\ell ^{ac}$-repr\'esentation  irr\'eductible enti\`ere  lisse  de $D^1$ (c'est  vrai pour $D^*, GL_2(F), SL_2(F)$).
 \end{theorem}
 
\begin{proof}[Preuve]  C'est \'evident pour les caract\`eres de $D^1$.

Consid\'erons   les  repr\'esentations irr\'eductibles  lisses   de $D^1$  de dimension $>1$. Donc $\ell \neq p$. Nous avons montr\'e en  \S\ref{ss:direct} qu'elles sont contenues dans les restrictions \`a $D^1$ des  repr\'esentations irr\'eductibles  lisses   $\Pi $ de $D^*$  de dimension $>2$.
Si  $\Pi$ est $\ell$-adique, on peut supposer que $ \Pi $ est enti\`ere 
(que son caract\`ere central est entier).   On rappelle  (voir \cite{V87})    que   $r_\ell (\Pi)$  est irr\'eductible car la dimension de $\Pi$ est $>2$,  que la correspondance de Langlands  Galois-Quaternions commute avec la r\'eduction modulo $\ell$,  et  que toute repr\'esentation  irr\'eductible lisse  $\ell$-modulaire de $D^*$ est la r\'eduction modulo $\ell$ d'une  repr\'esentation  irr\'eductible $\ell$-adique enti\`ere  lisse.  

Si $\ell \neq 2$, on note $X(\Pi)$ l'ensemble des caract\`eres  lisses $\ell$-adique $\chi$ de $F^*$ v\'erifiant \eqref{eq:QGG} et  $X(r_\ell(\Pi))$ l'ensemble des caract\`eres  lisses  $\ell$-modulaires de $F^*$ v\'erifiant \eqref{eq:QGG} pour  $r_\ell(\Pi))$. Les cardinaux  $|X(\Pi)|$ et 
$|X(r_\ell(\Pi))|$ sont diff\'erents si et seulement si $|X(\Pi)|=2$ et $|X(r_\ell(\Pi))|=4$  (voir \cite[Theorem 4.24]{HV25} et utiliser que la correspondance de Langlands Galois-Quaternions commute avec la torsion par des caract\`eres lisses). 
Par  \S \ref{ss:31}, on a:

Si $|X(\Pi)|=1$  donc $|X(r_\ell(\Pi))|=1$  alors la r\'eduction modulo $\ell$ de $\Pi|_{D^1} $ est irr\'eductible. 
 
 Si $|X(\Pi)|=2$, alors $\Pi|_{D^1} $ est somme de deux repr\'esentations irr\'eductibles non \`equivalentes   de  r\'eductions irr\'eductibles modulo $\ell$ non \'equivalentes si 
  $|X(r_\ell(\Pi))|=2$ et \'equivalentes si  $|X(r_\ell(\Pi))|=4$.
  
 Si $|X(\Pi)|=4$  donc  $|X(r_\ell(\Pi))|=4$, alors $\Pi|_{D^1} $ est somme  de deux repr\'esentations  \`equivalentes   de  r\'eductions irr\'eductibles modulo $\ell$. On en d\'eduit le th\'eor\`eme si $\ell \neq 2$. 
 
Si $\ell=2$ alors $p$ est impair. Si $p$ est impair,    le lemme \ref{le:srm}  pour $\mathbb Q_\ell^{ac}$ et $\mathbb F_\ell^{ac}$,  implique que $\Pi|_{D_1} $ est la somme de deux repr\'esentations de $D_1$ de r\'eduction modulo $\ell$ irr\'eductibles  et non \'equivalentes. On en d\'eduit le th\'eor\`eme si $p$ est impair, donc si $\ell =2$. 
  \end{proof}

  \section{Param\`etres de Langlands \'etendus}\label{CLE}
Soit $R$ un corps alg\'ebriquement clos de caract\'eristique  $\charf_R\neq p$. La correspondance de Langlands  pour $D^*$ (section de rappels) est une bijection qui associe \`a
 une $R$-repr\'esentation lisse irr\'eductible $\Pi $ de $D^*$, de dimension $>1$, une $R$-repr\'esentation lisse irr\'eductible
$\sigma(\Pi)$ de $W_F$, de dimension $2$, bien d\'efinie \`a isomorphisme pr\`es.
Cette bijection  est compatible \`a la torsion par les $R$-caract\`eres lisses de $F^*$, donc la classe de conjugaison
du morphisme $\overline \sigma_\Pi$ de $W_F$ dans $PGL_2(R)$ d\'eduit de $\sigma_\Pi$ ne d\'epend que de la restriction de $\Pi$ \`a $D^1$.
Ce morphisme est lisse, et elliptique au sens o\`u il ne prend pas ses valeurs dans un tore de $PGL_2(R)$. 
Comme tout tel morphisme se rel\`eve en une $R$-repr\'esentation lisse irr\'eductible de $W_F$ \cite{HV25}, on obtient
ainsi une bijection entre les $L$-paquets non triviaux pour $D^1$ et les classes de conjugaison de morphismes lisses elliptiques de $W_F$ dans $PGL_2(R)$. C'est ce qu'on peut appeler
la correspondance de Langlands pour $D^1$.
La question se pose alors de savoir si l'on peut \'etendre cette correspondance au sens o\`u, pour une $R$-repr\'esentation
lisse irr\'eductible $\Pi$ de dimension $>1$ de $D^*$, l'on peut   indexer les classes d'isomorphismes des composants irr\'eductibles de $\Pi|_{D^1}$ en termes de $\overline \sigma_\Pi$. Pour $R=\mathbb C$, voir \cite{HS12}, \cite{ABPS16}, \cite{AMPS17}. Pour $\charf_R\neq p$ nous avons \'etudi\'e la question analogue
pour $SL_2(F)$, et montr\'e que les r\'esultats pour $R=\mathbb C$ sont \'egalement valables si et seulement si  $\charf_R\neq 2$ \cite{HV25}.
Nous analysons ici le cas de $D^1$ pour toute caract\'eristique de $R$. Vogan  a sugg\'er\'e de consid\'erer non le centralisateur de $\overline \sigma_\Pi$
dans $PGL_2(R)$, comme nous l'avons fait pour $SL_2(F)$, mais plut\^ot son centralisateur $\mathcal C_\Pi$ dans $SL_2(R)$ (qui est fini), et les repr\'esentations de ce
centralisateur dont la restriction au centre de $SL_2(R)$ est fid\`ele.

Supposons d'abord $\charf_R\neq 2$. On peut d\'eterminer $\mathcal C_\Pi$.
Dans le  cas o\`u $d_\Pi=1$, qui ne se produit que pour $p=2$, l'image de $W_F$ dans $PGL_2(R)$ est isomorphe \`a $A_4$ ou $S_4$
et $\mathcal C_\Pi$ est le centre de $SL_2(R)$. 
Dans le  cas o\`u $d_\Pi=2$, l'image de $W_F$ dans $PGL_2(R)$ est un groupe di\'edral de cardinal
$2m>4$, et, quitte \`a conjuguer dans $PGL_2(R)$, le sous-groupe cyclique d'ordre $m$ est diagonal, de sorte que $\mathcal C_\Pi $ 
est engendr\'e par l'antidiagonale  $(1,-1) $ (voir le calcul dans \cite{HV25}), un \'el\'ement d'ordre $4$ de $SL_2(R)$. Ainsi $\mathcal C_\Pi$ a deux $R$-caract\`eres fid\`eles
qui devraient correspondre aux deux composants irr\'eductibles de $\Pi|_{D^1}$. Mais contrairement au cas de $SL_2(F)$, nous
ne disposons pas des mod\`eles de Whittaker pour choisir une des composantes irr\'eductibles de $\Pi|_{D^1}$.
M\^eme pour $R=\mathbb C$ un tel choix demande certainement des donn\'ees suppl\'ementaires, soit globales comme le mentionne
\cite[dernier paragraphe]{L24}, soit locales. Enfin dans le  cas $d_\Pi=4$, l'image de $W_F$ dans $PGL_2(R)$ est le groupe de Klein, et $\mathcal C_\Pi$
est le groupe quaternionien d'ordre $8$. Ce groupe a, \`a isomorphisme pr\`es, une seule
$R$-repr\'esentation irr\'eductible dont la restriction au centre est fid\`ele. Elle est de dimension $2$, ce qui interpr\`ete le fait que 
la multiplicit\'e du composant irr\'eductible de $\Pi|_{D^1}$ est $2$.

Supposons pour terminer que  $\charf_R=2$. Nous  avons  $d_\Pi=2$ et $\mathcal C_\Pi$
est cyclique d'ordre $2$, mais il ne poss\`ede qu'un seul $R$-caract\`ere, le caract\`ere trivial. On n'a donc pas, en ce cas,
une correspondance de Langlands \'etendue comme dans le cas $\charf_R\neq 2$
\section{Appendice:  Crit\`eres   d'irr\'eductibilit\'e}\label{irr}
 
Soit $R$ un corps commutatif, $G$ un groupe localement profini, $J$ un sous-groupe ouvert de $G$,  et $\lambda$ une $R$-repr\'esentation irr\'eductible lisse de $J$. On s'int\'eresse \`a l'irr\'eductibilit\'e de l'induite compacte $\ind_J^G\lambda$ de $\lambda$ \`a $G$  \cite{V96}.

\bigskip Regardons  les endomorphismes de $\ind_J^G\lambda$. L'induite compacte $\ind_J^G(\lambda) $ est contenue dans l'induite  $\Ind_J^G(\lambda)$ de $\lambda$ \`a $G$.
Les fonctions dans $\ind_J^G\lambda$ \`a support dans $JgJ$ pour $g\in G$ forment une sous-repr\'esentation $\ind_J^{JgJ} \lambda$ de $\ind_J^G(\lambda)|_J$.  On d\'efinit de m\^eme $\Ind_J^{JgJ} \lambda\subset \Ind_J^G(\lambda)|_J$. L'induction compacte  $\ind_J^G$ est l'adjointe \`a gauche de la restriction de $G$ \`a $J$, tandis que l'induction $\Ind_J^G$ est son adjointe \`a droite \cite[\S 5.7]{V96}. 

On a donc $\End_{RG}(\ind_J^G\lambda) \subset \Hom_{RG}(\ind_J^G\lambda, \Ind_J^G\lambda) $ avec:
\begin{equation}\label{Ii}\Hom_{RG}(\ind_J^G\lambda, \Ind_J^G\lambda) \simeq \Hom_{RJ}(\ind_J^G \lambda, \lambda ) =  \Hom_{RJ}(\oplus_{JgJ}\ind_J^{JgJ} \lambda,\lambda)
\end{equation}
$$\simeq \prod _{JgJ} \Hom_{RJ}(\ind_J^{JgJ} \lambda,\lambda)\simeq \prod _{JgJ} \Hom_{R(J\cap {}^g J)}( {}^g \lambda,  \lambda).$$
\begin{equation}\label{ii}\End_{RG}(\ind_J^G\lambda) \simeq \Hom_{RJ}(\lambda, \ind_J^G \lambda) =  \Hom_{RJ}(\lambda, \oplus_{JgJ}\ind_J^{JgJ} \lambda)
\end{equation}
$$\simeq \oplus_{JgJ} \Hom_{RJ}(\lambda,\ind_J^{JgJ} \lambda) \ \text{car $\lambda$ est irr\'eductible}.$$

Si $JgJ/J$ est compact, alors $ \ind_J^{JgJ} \lambda= \Ind_J^{JgJ} \lambda$ et par adjunction le dernier terme est isomorphe \`a $\oplus_{JgJ} \Hom_{R(J\cap {}^g J)}( \lambda, {}^g \lambda)$.
 
  La contribution de $\End_{RJ}\lambda  $ dans $\End_{RG}( \ind_J^G \lambda) $ est donn\'ee par l'induction compacte $\ind_J^G(\iota)$ o\`u $\iota\in \End_{RJ}\lambda$.

 \begin{lemma}\label{cri2} On suppose: 
  
  a)  Les propri\'et\'es \'equivalentes ci-dessous sont v\'erifi\'ees:
  
  $\bullet$ $\End_{RG}( \ind_J^G \lambda) = \End_{RJ}\lambda  $.
  
 $\bullet$   Le composant $\lambda$-isotypique de  $(\ind_J^G\lambda)|_J$ est \'egal \`a $\lambda$.

$\bullet$   $J=\{ g \in G \ | \ \Hom_{RJ}(\lambda, \ind_J^{JgJ}\lambda)\neq 0\}$. 
     
  b) Pour toute sous-repr\'esentation lisse $\pi$ de $\ind_J^G(\lambda)$, si $\lambda$ est quotient de $\pi|_J$ alors $\lambda$ est une sous-repr\'esentation de $\pi|_J$.
  
\noindent Alors la repr\'esentation $\ind_J^G \lambda$ est irr\'eductible.
   \end{lemma}  
     \begin{proof}[Preuve] On choisit une sous-repr\'esentation non-nulle $X$ de $\ind_J^G(\lambda)$. Comme $\ind_J^G(\lambda) \subset \Ind_J^G(\lambda)$, par adjunction, $\lambda$ est un quotient de $X|_J$. Par b), $\lambda$ est une sous-repr\'esentation de $X|_J$. 
  Donc il existe un endomorphisme de $\ind_J^G \lambda $ d'image non nulle $Y$ contenue dans $X$.  Par a), cet isomorphisme est de la forme $\ind_J^G(\iota)$ pour un endomorphisme non nul $\iota$ de $\lambda$. Comme $\lambda$ est irr\'eductible, $\iota \in \End_{RJ}\lambda$ est inversible, donc $\ind_J^G(\iota)\in \End_{RG}(\ind_J^G\lambda)$ est inversible et $\ind_J^G\lambda=Y=X $. Donc 
  $\ind_J^G(\lambda)$ est irr\'eductible.
  \end{proof}

Le crit\`ere simple d'irr\'eductibilit\'e   \cite[Lemma 3.2]{V00} lorsque   $R$ est alg\'ebriquement clos,  est un cas particulier  d'une variante \ref{var} du  lemme  \ref{cri2}.

  \begin{variant}\label{var} On suppose:  

 a') L"ensemble d'entrelacement de $\lambda$ dans $G$ est $J$.
   
  b') la propi\'et\'e b) o\`u l'on a permut\'e quotient et sous-repr\'esentation.  
  
  \noindent Alors la repr\'esentation $\ind_J^G \lambda$ est irr\'eductible.    \end{variant}
  On rappelle que  l"ensemble d'entrelacement de $\lambda$ dans $G$ est  $$\{ g\in G \ | \ \Hom_{R(J\cap {}^g J)}(\lambda,  {}^g  \lambda)\neq 0.\}.$$
 Il est \'egal \`a $J$  si et seulement si  (en conjuguant par  $g^{-1}$):

  $J= \{ g\in G \ | \ \Hom_{R(J\cap {}^g J)}( {}^g \lambda,  \lambda)\neq 0\}$
   si et seulement si    (par la formule \eqref{Ii}):

$\Hom_{RG}(\ind_J^G\lambda, \Ind_J^G\lambda)=\End_{RJ}\lambda$ qui implique (avec la formule \eqref{ii}):

$J=\{ g \in G \ | \ \Hom_{RJ}(\lambda, \ind_J^{JgJ}\lambda)\neq 0\}$ et $\Hom_{RG}(\ind_J^G\lambda, \Ind_J^G\lambda)=\End_{RG}( \ind_J^G \lambda) $.

\medskip 
    \begin{proof}[Preuve]   Soit $Y$ un quotient non nul de $\ind_J^G \lambda$. Par adjonction $\lambda$ est une sous-repr\'esentation de $Y|_J$. Par b') $\lambda$ est un quotient de $Y|_J$, donc par adjonction  $\Hom_{RG}(Y, \Ind_J^G \lambda) \neq 0$, et il existe un homomorphisme non nul de $\ind_J^G \lambda$ dans $\Ind_J^G\lambda$.  Son image est dans $\ind_J^G \lambda$. On termine comme dans la preuve du lemme \ref{cri2}.     \end{proof}

 Si $J$ est compact et de pro-ordre inversible dans $R$, la propri\'et\'e b) est \'evidente. Le crit\`ere suivant se passe de cette hypoth\`ese b).
 
Soit $J^1$ un sous-groupe compact de $J$ de pro-ordre inversible dans $R$ (de sorte que toutes les $R$-repr\'esentations lisses de $J^1$ sont semi-simples),  et $\mu$ une  $R$-repr\'esentation lisse irr\'eductible de $J^1$. 

Le lemme suivant est un raffinement du crit\`ere d'irr\'eductibilit\'e du lemme \ref{cri2}.  
\begin{lemma}\label{cri1} On suppose: 
 
c)  La   restriction de $\lambda$  \`a $J^1$ est $\mu$-isotypique de multiplicit\'e finie $m$.

d) Le composant $\mu$-isotypique de la   restriction   \`a $J^1$ de $\ind_J^G\lambda$ est de multiplicit\'e $m$.

\noindent Alors la repr\'esentation $\ind_J^G\lambda$ est irr\'eductible.

\end{lemma}
\begin{proof}[Preuve]  Si $\ind_J^G(\lambda)$  est r\'eductible, on choisit une sous-repr\'esentation non-nulle $X$ de $\ind_J^G(\lambda)$ telle que le quotient $Y$ soit non nul.  Par adjonction  $\lambda$ est une sous-repr\'esentation de $Y|_J$ et un quotient de   $ X|_J$. Donc la multiplicit\'e de $\mu$ dans la restriction \`a $J^1$ de $Y$ et de $X$ est au moins $2m$. Ceci contredit d).
 \end{proof}
 
 \begin{remark} \label{rem:irr} 
 
 $\bullet$   c) et  d)  impliquent  que le  composant $\mu$-isotypique de  $\ind_J^G\lambda|_{J^1}$ est la sous-repr\'esentation $\lambda$ de $\ind_J^G\lambda|_{J} $, donc c) et d) impliquent a). 
   
 $\bullet$  Si  $\lambda| _{J^1}$ est irr\'eductible (isomorphe \`a $\mu$, $m=1$) et son entrelacement dans $G$ est \'egal \`a $J$ alors  $\Hom_{RJ^1}(\mu, \ind_J^{JgJ^1}\lambda)=0$ si $g\not\in J$ donc 
   d) est v\'erifi\'e.
  
  \end{remark}
  
 \`A notre connaissance, un crit\`ere d'irr\'eductibilit\'e proche de celui du lemme \ref{cri1} est utilis\'e dans toutes les constructions explicites de  repr\'esentations irr\'eductibles cuspidales de groupes r\'eductifs $p$-adiques sur un corps commutatif $R$ de caract\'eristique  $\ell \neq p$ (sinon on n'a pas besoin de $J^1$) \cite{HV22}.  Il  est appliqu\'e avec  $J$ compact modulo le centre de $G$ et $J^1$ est un pro-$p$ sous-groupe ouvert de $J$. 
 
\section{Appendice: Commutateurs de $D^1$}\label{com}
Dans cet appendice, on consid\'ere un corps gauche $D$ de centre $F$ et de degr\'e $d^2$ sur $F$ and d'invariant de Hasse $r/d$ for $1\leq r\leq d, (r,d)=1$.  On utilise  les notations traditionnelles $O_D, P_D, k_D,  \rho, U_D$  d\'ej\`a utilis\'ees  quand $D$ est un corps de quaternions (Notations \ref{not}). On note $U^0_D=U_D, U^i_D=1+P_D^i$ pour $i>0$.
On choisit une racine de l'unit\'e $\omega\in D$ d'ordre $q^d-1$ et un g\'en\'erateur $p_D$ de $P_D$ tel que $p_D^d=p_F, \ p_D \omega p_D^{-1}=\omega^{q^r}$.  L'extension $E=F(\omega)$ de $F$ est non ramifi\'ee de degr\'e $d$.  On note $D^1$ le noyau de $\nrd$. \begin{theorem} Tout \'el\'ement de $D^1\cap U^1_D$ est produit de deux commutateurs de $D^1$.
\end{theorem}
 On en d\'eduit  le r\'esultat d\'eja connu (see \cite[Corollary  page 521]{Rie70}): 

\begin{corollary} Le groupe $(D^1,D^1)$ engendr\'e par les commutateurs de $D^1$ est  $D^1\cap U^1_D$. 
\end{corollary}
\begin{proof}[Preuve] Par le th\'eor\`eme, $ D^1\cap U^1_D \subset (D^1, D^1)$. Par l'application quotient $\varrho : O_D \to k_D$,  $D^1/(D^1\cap U^1_D)$ s'identifie au noyau $k^1_D$ de la norme de $k_D^*$ vers $k_F^*$,  et $k_D^1$ est cyclique. Donc $(D^1, D^1)\subset D^1\cap U^1_D$.  \end{proof}

\begin{remark} 
Si $H$ est un $F$-groupe simplement connexe et isotrope alors $H(F) = (H(F),H(F))$  by \cite[6.15]{PR84}. 
\end{remark}
Nous d\'emontrons maintenant le th\'eor\`eme\footnote{Nakayama and Matsushima \cite{NM43} ont montr\'e que chaque \'el\'ement of $D^1$ est le produit de trois \'el\'ements of $D^*$. Nous nous inspirons de leur preuve.}.
\begin{proof}[Preuve]
On fixe une    racine de l'unit\'e $z\in E$ d'ordre $(q^d-1)/(q-1)$.  
Pour $i>0$ et $x\in O_D$, on a $z (1+xp_D^i)z^{-1}= 1+yp_D^i$ avec $y\in O_D$ et $\rho(y)=\rho(z) \rho(z)^{-q^{ri}} \rho(x)$.

\medskip a) Supposons que $d$ ne divise pas $i$. Alors les \'el\'ements $t\in k_D$ tels que $t^{q^{ri}}=t$ forment une sous-extension de $k_D/k_F$ de degr\'e $e<d$. Comme \begin{equation} q^{d-1}-1 < (q^d-1)/(q-1), 
\end{equation}
car $(q-1)(q^{d-1}-1) < (q-1)q^{d-1}<q^d-q^{d-1} < q^d-1$, on a  $z^{q^{ri}}\neq z$. Donc 
 pour $s\in O_D$ il existe  $x\in O_D$ tel que le commutateur $(z, 1+xp_D^i)$ est $ 1+sp_D^i $ modulo $P_D^{i+1}$. 
 
 La norme r\'eduite envoie $U_D$ dans $U_F$, et $U^i_D $ dans $U^i_D \cap U_F$. 
Il est bien connu que $\nrd (U^i_E)=U^i_F$ pour $i\geq 0$ \cite[Chapitre V Proposition 1]{S68}. On en d\'eduit que l'on a 
\begin{equation}\label{eq:iso0}\nrd(U^i_D)=U^k_F \ \ \ \text{si $i=dk-j$ avec $0\leq j\leq d-1$}. 
\end{equation}
Il s'ensuit  que l'injection de $D^1$ dans $D^*$ induit pour $i>0$, un isomorphisme  
\begin{equation}\label{eq:iso} (D^1\cap U^i_D)/(D^1\cap  U^{i+1}_D)\to  U^i_D/U^{i+1}_D.
\end{equation}
Pour $u\in D^1 \cap U^i_D$ tel que $u \equiv 1+xp_D^i$ modulo $p_D^{i+1}$,  le commutateur $(z,u)$ est  $ 1+sp_D^i $ modulo $P_D^{i+1}$. 

\medskip b) Supposons que $i=kd$. Pour $x\in O_D$, le commutateur $(1+p_D, 1+xp_D^{i-1}) $ est $1+yp_D^i$ avec $\rho(y) =\rho(x)^{q^r}-\rho(x)$ \footnote{On calcule $(1+a)(1+b)(1+a')(1+b')$ avec $a+a'+aa'=b+b'+bb'=0$, en 
n\'egligeant les termes avec deux $a$ et un $b$, ou un $a$ et deux $b$. On  a $(1+a)(1+b)(1+a')=1+(1+a)b  (1+a')=1+b+ab+ba'=1+b+ab-ba$ et $(1+ b+ab-ba)(1+b')=1+ab-ba$. On prend  $a=p_D, b= xp_D^{i-1}$.} Tout \'el\'ement de $k_D$ de trace nulle dans $k_F$ est de la forme $t^{q^r}-t$ pour $t\in k_D$ \footnote{L'application lin\'eaire $t\mapsto t^{q^r}-t: k_D\to k_D$ de noyau  $k_F$ car $(r,d)=1$, a son image  contenue dans le noyau de la trace  $k_D\to k_F$ qui est  surjective.}. Pour $y \in O_D$ avec $ \rho(y) $ de trace nulle dans $k_F$, il existe donc $x\in O_D$  tel que $(1+p_D, 1+xp_D^{i-1})\equiv  1+yp_D^i $ modulo $p_D^{i+1}$.

\medskip c) Le polyn\^ome caract\'eristique r\'eduit de $p_D$ est $X^d-p_F$. Celui de $1+p_D$ est donc
$(X-1)^d-p_F$, donc  $\nrd(1+p_D)= 1+(-1)^{d-1}p_F$. Il existe $u\in 1+P_E$ tel que $\nrd(u)=\nrd(1+p_D)$. Fixons $u$ et posons $t= (1+p_D) u^{-1}$. Alors $ t\in D^1$ et $t\equiv 1+p_D$ modulo $p_D^2$.

Nous allons montrer (adaptant la m\'ethode de \cite{NM43}) par approximations successives : 
\begin{equation}\text{Pour tout $a\in D^1\cap U_D^1$, il existe $b,c\in D^1$ tels que $a=(z,b)(t,c)$.}
\end{equation}
Pour $i>0$ supposons trouv\'es $b_i,c_i \in D^1\cap U_D^1$ tels que 
$a\equiv (z,b_i)(t,c_i)$ modulo $U^i_D$. Pour $i=1$ on prend $b_i=c_i=1$. 

Si $ i $ n'est pas multiple de $d$ et $a=u(z,b_i)(t,c_i) $ pour  
$u\in  D^1 \cap U^i_D$, par a) il existe $v\in  D^1 \cap U^i_D$ tel que $(z,v)\equiv u$ modulo $U^{i+1}_D$. Comme  $(z,vb_i)=(z,v)(v,(z,b_i))(z,b_i)$ et  $(v,(z,b_i))\in U^{i+1}_D$, on a 
$a\equiv  (z,vb_i)(t,c_i)$ modulo  $U^{i+1}_D$. On pose $b_{i+1}=vb_i , c_{i+1}=c_i$.

Si $i=kd$ pour un entier $k>0$ et $a=(z,b_i)(t,c_i)u$ pour  $ u\in  D^1 \cap U^i_D$, par b) il existe $v\in  D^1\cap  U^{i-1}_D$ tel que $u\equiv (t,v) $ modulo $U^{i+1}_D$.
Donc $a\equiv (z,b_i)(t,c_i)(t,v) $modulo $U^{i+1}_D$. Comme  $(t,c_iv)=(t,c_i)(c_i,(t,v))(t,v)$ et  $ (c_i,(t,v))\in  U^{i+1}_D$, on a 
$a\equiv (z,b_i)(t,c_iv)$ modulo $U^{i+1}_D$. On pose alors $b_{i+1}=b_i ,c_{i+1}=c_iv$.
Les suites $(b_i)$ et $(c_i)$ convergent, vers $b$ et $c$ respectivement, et \`a la limite on obtient $a=(z,b)(t,c)$.  
\end{proof}

 \section{Appendice:  Le cas $R=\mathbb F_p^{ac}, F = \mathbb Q_p$. }\label{modp}

 Les r\'esultats  de \S \ref{ss:31} implique lorsque $R=\mathbb F_p^{ac}$ et $F=\mathbb Q_p$ la proposition suivante.
 
\begin{proposition}\label{pro:D}  Il existe des  $\mathbb F_p^{ac}$-repr\'esentations irr\'eductibles lisses  non \'equivalentes  
$\Pi_0, \ldots, \Pi_{p-1}$ de $D^*$ et $\pi_0, \ldots, \pi_{p-1}$ de $D^1$ telles que:

1)  $\Pi_r = \ind_{\mathbb Q_p^* U_D}^{D^*}\lambda_r , \ \ \   \ \ 0\leq r \leq p-1,$ 

\noindent o\`u  $\lambda_r$ est un  $\mathbb F_p^{ac}$-caract\`ere de $U_D$ de restriction $\pi_r= \pi_0^{r+1}$  \`a $ D^1$, pour un plongement $\pi_0$ de $k_D^1$ dans $(\mathbb F_p^{ac})^*$ relev\'e \`a $D^1$,  que l'on \'etend   \`a $\mathbb Q_p^*U_D$  en envoyant $p$ sur l'identit\'e.

2)  Les $\mathbb F_p^{ac}$-repr\'esentations irr\'eductibles lisses de $D^1$ non triviales sont  les caract\`eres $\pi_r$ pour $0\leq r \leq p-1$. 
 
3) $ \Pi_r |_{D^1}\simeq  \pi_r \oplus \pi_{p-1-r}, \ \ 0\leq r \leq p-1 $ (car $\pi_r^p=\pi_0^{(r+1)p}=\pi_0^{p-r}=\pi_{p-1-r}$).
\end{proposition}

  Ceci rappelle des  r\'esultats analogues pour  $GL_2(\mathbb Q_p)$   et  $SL_2(\mathbb Q_p)$  \cite[Th\'eor\`eme 4.12]{Abde14}.
Une $\mathbb F_p^{ac}$-repr\'esentation irr\'eductible  supercuspidale de $GL_2(\mathbb Q_p)$ ou de $SL_2(\mathbb Q_p))$ est appel\'ee supersinguli\`ere. On note  $ \Sym^r (\mathbb F_p^2) $    la repr\'esentation de $GL_2(\mathbb Z_p)$    \'etendue \`a $K={\mathbb Q_p}^* GL_2(\mathbb Z_p)$ en envoyant $p$ sur l'identit\'e.   

\begin{proposition}\label{pro:GL} Il existe   des repr\'esentations supersinguli\`eres non \'equivalentes
$\Pi_0, \ldots, \Pi_{p-1}$ de $GL_2(\mathbb Q_p)$ et $\pi_0, \ldots, \pi_{p-1}$ de $SL_2(\mathbb Q_p)$  telles que:

1) $\Pi_r = (\ind_{K}^{GL_2(\mathbb Q_p)}\Sym^r (\mathbb F_p^2)) / T , \ \ \   \ \ 0\leq r \leq p-1,$ 

\noindent   est la repr\'esentation conoyau d'un certain endomorphisme $T$ de $\ind_{K}^{GL_2(\mathbb Q_p)} \Sym^r (\mathbb F_p^2) $.  

2)  Une repr\'esentation supersinguli\`ere de $SL_2(\mathbb Q_p)$ est  isomorphe \`a une repr\'esentation $\pi_r$ pour un unique $0\leq r \leq p-1$.

3)  $  \Pi_r |_{SL_2(\mathbb Q_p)}\simeq  \pi_r \oplus \pi_{p-1-r}, \ \ 0\leq r \leq p-1. $ 

\end{proposition}  

La repr\'esentation a multiplicit\'e $1$ sauf si  $p$ est impair et $r= (p-1)/2$ o\`u elle est $2$.

\bigskip On comparera ces propositions avec les r\'esultats de   \cite{BS25}. La restriction \`a $SL_2(\mathbb Q_p)$ d'une  repr\'esentation de Banach unitaire $p$-adique irr\'eductible  $\Pi$  of $GL_2(\mathbb Q_p)$ sur une extension finie $E/\mathbb Q_p$, associ\'ee par la correspondance 
locale de  Langlands $p$-adique, \`a une repr\'esentation galoisienne $\sigma_\Pi$ absolument  irr\'eductible continue de dimension $2$,  est une 
somme directe  de  $s\leq 2$ repr\'esentations irr\'eductibles.
Elle est sans multiplicit\'e et  $s$ est  le cardinal  $S$  du centralisateur dans  $PGL_2$ de la repr\'esentation projective galoisienne   associ\'ee \`a  $\sigma_\Pi$  sauf si 
 $\sigma_\Pi$ est triplement imprimitive auquel cas $S= 4$  et
 $\Pi|_{\SL_2(\mathbb Q_p)}$ est 
somme directe de  deux repr\'esentations  irr\'eductibles \'equivalentes. 

 \def\refname{R\'ef\'erences}
    
    \end{document}